\documentclass[12pt]{amsart}

\usepackage{amsmath,amssymb,latexsym} 
\usepackage{fullpage}
\usepackage{etoolbox}
\usepackage{enumitem}
\usepackage{color}
\usepackage[pagebackref=true,colorlinks,linkcolor=blue,citecolor=magenta]{hyperref}
\usepackage[colorinlistoftodos,bordercolor=orange,backgroundcolor=orange!20,linecolor=orange,textsize=scriptsize]{todonotes}
\usepackage{frenchineq}
\usepackage{soul}

\newtheorem{theorem}{Theorem} 
\newtheorem{corollary}[theorem]{Corollary}
\newtheorem{proposition}[theorem]{Proposition}

\newtheorem{lemma}[theorem]{Lemma}

\theoremstyle{definition}

\def\Ker{\operatorname{Ker}}
\def\HH{\mathcal{H}}

\def\E{\mathcal{E}}
\def\U{\mathcal{U}}
\def\V{\mathcal{V}}
\def\W{\mathcal{W}}

\def\R{\mathbb{R}}

\def\Z{\mathbb{Z}}
\def\F{\mathbb{F}}
\def\FF{\mathcal{F}}

\def\S{\mathcal{S}}
\def\C{\mathbb{C}}

\def\xx{\boldsymbol{x}}
\def\yy{\boldsymbol{y}}
\def\zz{\boldsymbol{z}}

\def\ds{\displaystyle}

\def\KG{\operatorname{KG}}
\def\cd{\operatorname{cd}}
\def\rk{\operatorname{rk}}

\def\V{\mathcal{V}}
\def\E{\mathcal{E}}

\def\ind{\operatorname{ind}}
\def\coind{\operatorname{coind}}
\def\Xind{\operatorname{Xind}}
\def\Hom{\operatorname{Hom}}
\def\NP{\operatorname{NP}}
\def\coNP{\operatorname{coNP}}

\def\span{\operatorname{span}}

\def\minrk{\operatorname{minrk}}

\def\od{\operatorname{\overline{\xi}}}

\begin{document} 

\title{Topological bounds for graph representations \\ over any field}

\author{Meysam Alishahi}
\address{M. Alishahi,
Faculty of Mathematical Sciences,
Shahrood University of Technology, Shahrood, Iran}
\email{meysam\_alishahi@shahroodut.ac.ir}

\author{Fr\'ed\'eric Meunier}
\address{F. Meunier, Universit\'e Paris Est, CERMICS, 77455 Marne-la-Vall\'ee CEDEX, France}
\email{frederic.meunier@enpc.fr}

\begin{abstract}
Haviv ({\em European Journal of Combinatorics}, 2019) has recently proved that some topological lower bounds on the chromatic number of graphs are also lower bounds on their orthogonality dimension over $\R$. We show that this holds actually for all known topological lower bounds and all fields. We also improve the topological bound he obtained for the minrank parameter over $\R$ --  an important graph invariant from coding theory -- and show that this bound is actually valid for all fields as well. The notion of independent representation over a matroid is introduced and used in a general theorem having these results as corollaries. Related complexity results are also discussed.
\end{abstract}


\maketitle 

\section{Introduction}\label{sec:intro}


Given a field $\F$, an assignment of a vector $\xx_v\in\F^t$ to each vertex $v$ of a graph $G$ is a $t$-dimensional {\em orthogonal representation} of $G$ over $\F$ if $\langle\xx_u,\xx_u\rangle\neq 0$
for every vertex $u$ and $\langle\xx_u,\xx_v\rangle=0$ for every edge $uv$. Here $\langle\yy,\zz\rangle$ stands for $\sum_iy_iz_i$ except when the field is $\C$ for which we can also understand $\langle\yy,\zz\rangle$ as $\sum_iy_i\overline{z_i}$; none of our results depends on which definition is chosen for complex vector spaces.

The {\em orthogonality dimension} of a graph $G$ over $\F$, denoted by $\od_{\F}(G)$, is the smallest $t$ for which there exists a $t$-dimensional orthogonal representation of $G$ over $\F$. This definition is sometimes given for the complement of $G$. We follow the setting proposed by Haviv~\cite{haviv2019approximating}, 
whose work~\cite{haviv2019topological} was the starting point of this paper.

The orthogonality dimension of a graph over $\R$ was originally introduced by Lov\'asz~\cite{lovasz1979shannon} for his work on the Shannon capacity of a graph. 
Peeters~\cite{Peeters1996} extended the notion over any field, and it is now a well-studied notion. The orthogonality dimension of a graph $G$ is always at least its clique number and at most its chromatic number, i.e., $\chi(G)\geq\od_{\F}(G)\geq\omega(G)$.

The clique number is one of the most natural lower bounds on the chromatic number. After the celebrated work of Lov\'asz on the chromatic number of Kneser graphs~\cite{lovasz1978kneser}, topology has also been a way to get lower bounds on this parameter. Recently, Haviv~\cite{haviv2019topological} proved that, similarly, two of these topological bounds, the ``2-colorability defect of a Kneser representation of $G$'' denoted by $\cd_2(\FF_G)$ and the so-called ``B\'ar\'any's bound'', are lower bounds on the orthogonality dimension over $\R$.

Another geometric parameter associated to a graph is the minrank. A $V\times V$ matrix $A$ {\em represents} a graph $G=(V,E)$ over a field $\F$ if $A_{u,u}\neq 0$ for every vertex $u$ and if $A_{u,v}=0$ for every distinct nonadjacent vertices $u$ and $v$. The {\em minrank} of $G$ over $\F$ is defined as $\minrk_{\F}(G)=\min\{\rk_{\F}(A)\colon A\mbox{ represents }G\mbox{ over }\F\}$. Haviv also proved the inequality $\minrk_{\R}(\overline G)\geq\sqrt{\frac{\cd_2(\mathcal{F}_G)} 2}$ (where $\overline G$ is the complement of $G$). 

We improve both results of Haviv. First, we prove that whatever the field, {\em all} known topological lower bounds on the chromatic number are also lower bounds on the orthogonality dimension. Second, we prove that the square root is not necessary in the lower bound of the minrank, and that this lower bound is again valid for any field $\F$. These improvements are consequences of our main result, which we state in the more general setting of matroids to emphasize the role played by independence. 

Let $G$ be a graph and $M$ a matroid. Consider an assignment of an element $x_v\in M$ to each vertex $v$ of $G$ such that $x_v$ is not in the span of the elements assigned to the neighbors of $v$, i.e., $x_v\notin\span(\{x_u\colon u\in N(v)\})$ for every vertex $v$. In particular, $x_v$ is never a loop. We call such an assignment an {\em independent representation} of $G$ over $M$. The special case when $M$ is the linear matroid built from a vector space $\U$ -- which we call then also an {\em independent representation} of $G$ over $\U$ -- was already studied in the context of linear index codings; see the beginning of Section C of the paper by Shanmugam, Dimakis, and Langberg~\cite{shanmugam2013local}. This special case is also close to the notion of independence-preserving representation introduced by Lov\'asz and Vesztergombi~\cite[Chapter 10]{lovasz1999geometric}.

Our main result is the following theorem which is a counterpart for independent representations of the ``zig-zag'' theorem by Simonyi and Tardos~\cite{simonyi2006local} and of its generalizations~\cite{alishahi2017strengthening,simonyi2013colourful}. The quantity $\Xind(\Hom(K_2,G))$, which appears in its statement, is the ``cross-index of the Hom-complex''  that provides a lower bound on the chromatic number that dominates all known topological lower bounds on it~\cite[Section 3]{simonyi2013colourful}.
Its definition is rather technical and is postponed to a section where we discuss these topological bounds. We simply emphasize that $\Xind(\Hom(K_2,G))+2$ is nonsmaller than $\cd_2(\FF_G)$ and nonsmaller than ``B\'ar\'any's bound''.

\begin{theorem}\label{thm:matr}
Let $G$ be a graph with at least one edge and let $t=\Xind(\Hom(K_2,G))+2$. For every independent representation of $G$ over a matroid $M$, there is a complete bipartite subgraph $K_{\lfloor t/2\rfloor,\lceil t/2\rceil}$ such that the elements assigned to every side are independent, and thus in particular the following inequality holds: $\rk M\geq\frac 1 2 \Xind(\Hom(K_2,G))+2$. 
\end{theorem}

The bound on $\rk M$ is a direct consequence of the first part of the sentence: the elements assigned to the side with $\lceil t/2\rceil$ vertices are independent and any element on the other side is not in their span.
%

Our improvements over Haviv's results are formulated as the following corollaries.

\begin{corollary}\label{cor:ortho}
For every graph $G$ with at least one vertex and every field $\F$, the following inequality holds: $\ds{\od_{\F}(G)\geq\Xind(\Hom(K_2,G))+2}$.
\end{corollary}

\begin{proof}
Consider an $s$-dimensional orthogonal representation over $\F$. It is an independent representation over $\F^s$. Let $t=\Xind(\Hom(K_2,G))+2$.

If $G$ has no edge, then $t=1$ by definition (see Section~\ref{subsubsec:cross}) and the inequality to be proved is obvious. 

So suppose that $G$ has at least one edge. According to Theorem~\ref{thm:matr}, there is a complete bipartite subgraph $K_{\lfloor t/2\rfloor,\lceil t/2\rceil}$ such that the vectors assigned to every side are linearly independent. Denote by $\U$ and $\W$ the vector spaces spanned by the vectors assigned respectively to one side of the complete bipartite graph and to the other. Since we are given an $s$-dimensional representation, we have $s=\dim\U+\dim\U^{\bot}$. (This property is true as soon as we work with a nondegenerate bilinear form on a finite-dimensional vector space.) This representation being orthogonal, we have $\W\subseteq\U^{\bot}$. Combined with the previous equality, it leads to $s\geq\dim\U+\dim\W=t$.
\end{proof}

Corollary~\ref{cor:ortho} shows in particular that if the chromatic number of a graph matches one of the topological lower bounds known in the literature, then the orthogonal dimension of this graph does not depend on the field and is equal to the chromatic number. There are many graphs like this, such as Kneser graphs, Schrijver graphs, some Mycielski graphs, Borsuk graphs;  see~\cite{matousek2002topological} for the definition of these graphs and other examples. 

\begin{corollary}\label{cor:minrk}
For every graph $G$ with at least one edge and every field $\F$, the following inequality holds: $\ds{\minrk_{\F}(\overline G)\geq\frac 1 2 \Xind(\Hom(K_2,G))+2}$.
\end{corollary}

\begin{proof}
Consider an independent representation over $\F^s$. According to Theorem~\ref{thm:matr}, there is a complete bipartite subgraph $K_{\lfloor t/2\rfloor,\lceil t/2\rceil}$ such that the vectors assigned to every side are linearly independent. The vectors assigned to the side with $\lceil t/2\rceil$ vertices are independent, and any vector on the other side is not in their span. Thus, $s\geq t/2+1$.

The conclusion comes from the following fact: the minimum $s$ such that there exists an independent representation of a graph over $\F^s$ is the minrank of its complement over $\F$; see Lemma~\ref{lem:ind-minrk} in Section~\ref{subsec:rep-mat}.
\end{proof}

\section{Tools}

\subsection{Topological lower bounds on the chromatic number} 

\subsubsection{The $2$-colorability defect lower bound}\label{subsubsec:cd} Given a hypergraph $\HH=(\V,\E)$, the graph $\KG(\HH)$ has $\E$ as vertex set, and two vertices of $\KG(\HH)$ are adjacent if the corresponding edges in $\HH$ are disjoint. The hypergraph $\HH$ is a {\em Kneser representation} of $G$ if $G$ and $\KG(\HH)$ are isomorphic. All simple graphs have a Kneser representation. The {\em $2$-colorability defect} of $\HH$ is the minimum vertices to remove from $\HH$ so that it becomes $2$-colorable. Dol'nikov~\cite{dol1988certain} proved that we always have $\chi(G)\geq\cd_2(\HH)$ when $\HH$ is a Kneser representation of $G$. This is one of the two topological lower bounds used in the paper by Haviv.

\subsubsection{B\'ar\'any's bound} B\'ar\'any~\cite{barany1978short} proposed an alternate way to compute the chromatic number of Kneser graphs. As noted by Matou\v{s}ek and Ziegler~\cite{matousek2002topological}, his argument relies on the following claim: {\em Let $\HH=(\V,\E)$ be a Kneser representation of a graph $G$. If $\V$ can be placed into the $(t-2)$-dimensional sphere in such a way that for every open hemisphere $H$ there exists an edge $e\in\E$ with $e\subseteq H$, then $\chi(G)\geq t$.} (Kneser representations have been defined in Section~\ref{subsubsec:cd}.) This is the other topological lower bound used  by Haviv.

\subsubsection{The cross-index lower bound}\label{subsubsec:cross} A {\em free $Z_2$-poset} is a poset with a fixed-point free involution $\nu$ that preserves the order. A map $\phi\colon P\rightarrow Q$ between two free $Z_2$-posets
$P$ and $Q$ is an {\em order-preserving $Z_2$-map} if it is order-preserving and if $\phi(\nu\cdot p)=\nu\cdot\phi(p)$ for all $p\in P$. 
For a nonnegative integer $n$, we define $Q_n$ to be the free $Z_2$-poset with elements $\{\pm 1,\ldots,\pm (n+1)\}$ and whose partial order $<_{Q_n}$ is given by $p<_{Q_n}q$ if $|p|<|q|$. Note that $-i$ and $+i$ are not comparable in $Q_n$. For a free $Z_2$-poset $P$, the {\em cross-index} of $P$, denoted by $\Xind(P)$, is the minimum $n$ such that there is an order-preserving $Z_2$-map from $P$ to $Q_n$. Note that if $P$ is nonempty, then $\Xind(P)\geq 0$. By convention, the value of $\Xind(P)$ is $-1$ when $P$ is empty.

The {\em Hom complex} of a graph $G=(V,E)$, denoted by $\Hom(K_2, G)$, is a free $Z_2$-poset consisting of all pairs $(X,Y)$ such that $X$ and $Y$ are nonempty disjoint subsets of $V$ and $G[X,Y]$ is a complete bipartite graph. The partial order of this poset is the inclusion $\subseteq$ extended to pairs of subsets of $V$: the relation $(X,Y)\subseteq (X',Y')$ holds when $X\subseteq X'$ and $Y\subseteq Y'$. The involution $\nu$ on $\Hom(K_2, G)$ is defined by $\nu\cdot(X,Y)=(Y,X)$.

Simonyi, Tardif, and Zsb\'an~\cite{simonyi2013colourful} proved that we always have $\chi(G)\geq\Xind(\Hom(K_2,G))+2$.

\subsubsection{Comparison} Simonyi, Tardif, and Zsb\'an also explain why this provides a better lower bound than the other topological lower bounds like $\ind(B_0(G))$ or $\coind(B_0(G))$; see for instance~\cite{matousek2002topological,simonyi2013colourful} for the definitions. Matou\v{s}ek and Ziegler~\cite{matousek2002topological} have provided a hierarchy on the various topological lower bounds on the chromatic number. This shows that the cross-index lower bound of Simonyi, Tardif, and Zsb\'an is the best topological lower bound. In particular, we have $\Xind(\Hom(K_2,G))+2\geq\cd_2(\HH)$ when $\HH$ is a Kneser representation of $G$, and the same inequality with B\'ar\'any's bound.

\subsection{A Fan lemma for the cross-index}

There is a ``Fan lemma'' for the cross-index, proved by the present two authors and Hossein Hajiabolhassan~\cite[Lemma 4]{alishahi2017strengthening}.

\begin{lemma}\label{lem:fan-xind}
Let $(P,\preceq)$ be a free $Z_2$-poset and $r=\Xind(P)+1$. Consider a map $\phi\colon P\rightarrow \Z\setminus\{0\}$ such that
\begin{itemize}
\item $\phi(p)=-\phi(q)$ implies that $p$ and $q$ are not comparable.
\item $|\phi(p)|\leq|\phi(q)|$ when $p\preceq q$.
\item $\phi(\nu\cdot p)=-\phi(p)$ for every $p$.
\end{itemize}
Then there exists a chain $p_1\prec p_2\prec\cdots\prec p_r$ such that $$0<-\phi(p_1)<+\phi(p_2)<-\phi(p_3)<\cdots<(-1)^{r}\phi(p_r).$$
\end{lemma}

\subsection{Minrank and independent representations}\label{subsec:rep-mat}

In the proof of Corollary~\ref{cor:minrk}, we use the following fact, which seems to be common knowledge in the coding theory community (and which is used implicitly in the paper by Shanmugam, Dimakis, and Langberg~\cite{shanmugam2013local}). It has been communicated to us together with a proof by Ishay Haviv.

\begin{lemma}\label{lem:ind-minrk}
The minimum $s$ for which there exists an independent representation of a graph $G$ over $\F^s$ is the minrank of the complement of $G$ over $\F$.
\end{lemma}

\begin{proof}
Let $r=\minrk_{\F}(\overline G)$ and $s$ the minimum integer such that there exists an independent representation of $G$ over $\F^s$.

Consider a matrix $A$ of rank $r$ representing $\overline G$.
Without loss of generality, assume that the first $r$ columns of $A$ are linearly independent. 
Let $B$ be the submatrix of $A$ whose columns are the first $r$ columns of $A$. Assign to each vertex $v$ the $v$ row of $B$, say $\xx_v^T$.  Suppose for a contradiction that there is a vertex $v$ such that $\xx_v=\sum_{u\in N(v)}\lambda_u\xx_u$ for some $\lambda_u$'s in $\F$. Let $\yy$ be the $v$ column of $A$. 
Since $\yy$ can be written as a linear combination of the columns of $B$, there is a vector $\zz$ such that $B\zz=\yy$.  
Note that $\yy$ has $0$ on its entries in $N(v)$ and nonzero on its $v$ entry. Hence, 
$$0\neq \langle\xx_v,\zz\rangle=\sum_{u\in N(v)}\lambda_u\langle\xx_u,\zz\rangle=0,$$ 
which is a contradiction. Therefore, $r\geq s$.

Conversely, consider an independent representation $(\xx_v)_{v\in V}$ of $G=(V,E)$ over $\F^s$. Let $B$ be the $V\times [s]$ matrix whose row $v$ is $\xx_v^T$. Since $\xx_v$ is not in the span of the $\xx_u$'s with $u\in N(v)$, there exists for each $v$ a vector $\yy_v$ in $\F^s$ such that $B\yy_v$ is $0$ on its entries in $N(v)$ and nonzero on its $v$ entry: look at the submatrix $C$ corresponding to the rows in $N(v)\cup\{v\}$; the rank of $C$ is larger than the rank of the submatrix $D$ without the $v$ row and we can choose $\yy_v$ in $\Ker D\setminus\Ker C$. The $V\times V$ matrix whose $v$ column is $B\yy_v$ represents $\overline G$ over $\F$ and its rank is at most $s$. Therefore, $s\geq r$.
\end{proof}

\section{Proof of the main result}

\begin{proof}[Proof of Theorem~\ref{thm:matr}]
Consider an independent representation $(x_v)_{v\in V}$ of $G=(V,E)$ over a matroid $M$. For $U\subseteq V$, denote by $\S(U)$ the span generated by the elements $x_v$ with $v\in U$.
We consider the set of all possible such $\S(U)$'s. We put on this set an arbitrary total order $\preceq$ that refines the order induced by the ranks, i.e., which is such that $\S(U)\preceq\S(U')$ if $\rk\S(U)\leq\rk \S(U')$.

We introduce now a map $\phi$ from $\Hom(K_2,G)$ to $\Z\setminus\{0\}$. Let $(X,Y)\in\Hom(K_2,G)$. Since $X$ and $Y$ are both nonempty, we have $\S(X)\neq\S(Y)$: any $x_v$ with $v\in X$ is not in $\S(Y)$. We can thus
define unambiguously $\phi$ as 
$$\phi(X,Y)=\left\{\begin{array}{ll}+(\rk\S(X)+\rk\S(Y)) & \mbox{if $\S(X)\preceq\S(Y)$.} \\ -(\rk\S(X)+\rk\S(Y)) & \mbox{if $\S(Y)\preceq\S(X)$.}\end{array}\right.$$ 

We have $\phi(Y,X)=-\phi(X,Y)$. For $(X,Y)\subseteq (X',Y')$, we have $|\phi(X,Y)|\leq|\phi(X',Y')|$, and if $|\phi(X,Y)|=|\phi(X',Y')|$, then $\rk\S(X)=\rk\S(X')$ and $\rk\S(Y)=\rk\S(Y')$, which implies that $\S(X)=\S(X')$ and $\S(Y)=\S(Y')$, and hence $\phi(X,Y)=\phi(X',Y')$. The map $\phi$ satisfies the condition of Lemma~\ref{lem:fan-xind}. There exists thus in $\Hom(K_2,G)$ a sequence $(X_1,Y_1)\subseteq\cdots\subseteq(X_{t-1},Y_{t-1})$ such that the 
$|\phi(X_i,Y_i)|$'s form an increasing sequence of positive numbers, and such that the signs of the $\phi(X_i,Y_i)$'s alternate, starting from a minus sign. Since $G$ has at least one edge, we have $t\geq 2$, i.e., this sequence is not empty. Because of this alternation, we have $\rk\S(X_{t-1})\geq\lfloor t/2\rfloor$ and $\rk\S(Y_{t-1})\geq\lceil t/2\rceil$. We can thus choose a set $X^*$ of $\lfloor t/2\rfloor$ vertices in $X_{t-1}$ such that the elements in $\{x_v\colon v\in X^*\}$ form an independent set, and a set $Y^*$ of $\lceil t/2\rceil$ vertices in $Y_{t-1}$ such that the elements in $\{x_v\colon v\in Y^*\}$ form an independent set. The complete bipartite subgraph of $G$ with bipartition $X^*,Y^*$ is the sought subgraph.
\end{proof}

\section{Deciding the existence of large colorful complete bipartite subgraphs}\label{sec:complexity}

When $M$ is the $r$-uniform matroid over $[r]$, the independent representations over $M$ are exactly the proper colorings with $r$ colors. For a properly colored graph $G$, Theorem~\ref{thm:matr} implies thus the existence of a complete bipartite subgraph with at least $\lfloor t/2\rfloor$ colors on one side and at least $\lceil t/2\rceil$ colors on the other side, where $t=\Xind(\Hom(K_2,G))+2$. This is almost the ``zig-zag'' lemma for the cross-index by Simonyi, Tardif, and Zsb\'an~\cite{simonyi2013colourful}: the statement about the alternation between the two sides is missing but we have indeed two sides of ``balanced'' sizes. Therefore, Theorem~\ref{thm:matr} can be seen as an extension of those results ensuring, via topological arguments, the existence of large complete bipartite subgraphs that are balanced and colorful.

Since these results provide a sufficient condition for the existence of such a colorful subgraph, it is natural to ask whether this existence can be efficiently decided when the condition is not satisfied. Though natural, this complexity question does not seem to have been addressed yet.

\begin{proposition}\label{prop:complex}
Let $G$ be a properly colored graph. Deciding whether there is a complete bipartite subgraph in $G$ with all colors is $\NP$-complete, even if $G$ is bipartite.
\end{proposition}

\begin{proof}
We reduce monotone 3SAT to our problem. Montone 3SAT is a special case of 3SAT where in each clause, the variables are either all negated or all unnegated. This is an $\NP$-complete problem; see \cite{garey2002computers}. We consider an instance of monotone 3SAT with clauses $C_1,\ldots,C_m$ and variables $z_1,\ldots,z_n$. We build the following bipartite graph $H=(X,Y;E)$:
\begin{itemize}
\item In $X$, we put all pairs $(i,j)$ such that $z_j$ is a variable of $C_i$.
\item In $Y$, we put all pairs $(i',j')$ such that $\bar z_{j'}$ is a variable of $C_{i'}$.
\item We connect a vertex $(i,j)$ in $X$ to a vertex $(i',j')$ in $Y$ if $j$ is distinct from $j'$.
\item We color each vertex $(i,j)$ by $i$. Since we are considering an instance of monotone 3SAT, this is a proper coloring with $m$ colors.
\end{itemize}
We claim that there is a complete bipartite subgraph with $m$ colors if and only if the instance is satisfiable. This is what we show now.

Suppose that the instance is satisfiable and consider a solution. Select in $X$ all $(i,j)$ such that $z_j$ is true. Select in $Y$ all $(i',j')$ such that $z_{j'}$ is false. Clearly this forms a complete bipartite subgraph where the $m$ colors are used (since each clause is satisfied).

Conversely, suppose that there is a complete bipartite subgraph where the $m$ colors are used. Denote by $K$ the vertex set of this complete bipartite subgraph. Assign true to each $z_j$ such that $(i,j)$ is in $X\cap K$ for some $i$. Assign false to each $z_{j'}$ such that $(i',j')$ is in $Y\cap K$ for some $i'$. There cannot be any contradiction in this way of assigning values to the variables since otherwise the two vertices 
$(i,j)$ and $(i',j)$ must appear in $K$ in distinct sides and be thus connected which is impossible. 
Since the $m$ colors are used, every clause is satisfied. Assign any values to the remaining variables to get a solution to the instance.
\end{proof}

Proposition~\ref{prop:complex} has the following corollary about deciding the existence of ``balanced'' colorful bipartite subgraphs, as in ``zig-zag''-type results.

\begin{corollary}\label{CoNP:balanced}
Let $G$ be a properly colored graph and $t$ an integer. Deciding whether there is a complete bipartite subgraph in $G$ with at least $\lfloor t/2\rfloor$ colors on one side and at least $\lceil t/2\rceil$ colors on the other side is $\NP$-complete, even if $G$ is bipartite.
\end{corollary}

\begin{proof}
We reduce the problem of Proposition~\ref{prop:complex} to this one. Let $r$ be the number of colors used in the proper coloring of $G$. For an integer $s$, we define $G_s$ to be the graph $G$ with $s$ extra dummy vertices connected to all vertices of $G$ and forming a stable set in $G_s$. Color each of these new vertices with a new color. It is clear that $G$ has a complete bipartite graph having all $r$ colors if and only if there is an $s\in \{0,\ldots,r-1\}$ for which $G_s$ has a complete bipartite subgraph with $\lfloor t/2\rfloor$ colors on one side and at $\lceil t/2\rceil$ colors on the other side, where $t=s+r$.
\end{proof}

\section{Further comments}

\subsection{Local chromatic number}\label{subsec:local}

Consider independent representations of a graph over the uniform matroid $U_m^r$ of rank $r$ with $m$ elements. They are exactly proper colorings with $m$ colors and at most $r-1$ colors in the neighborhood of every vertex. So, by considering arbitrary values for $m$, the bound on $\rk M$ provided by Theorem~\ref{thm:matr} covers the lower bound on the local chromatic number found by Simonyi and Tardos~\cite[Theorem 1]{simonyi2006local}. This also implies that the bound of Theorem~\ref{thm:matr} is tight in some cases, e.g., for some Schrijver  graphs; see~\cite[Theorem 3]{simonyi2006local}.

\subsection{Better bounds for matroids?}
A way to get the lower bound of Corollary~\ref{cor:ortho} for independent representations of a graph $G$ over a matroid $M$ is to consider independent representations over $M$ satisfying the following extra condition: 
\begin{equation}\label{cond}\tag{$\star$}
\text{$\rk \S(X)+\rk \S(Y)\leq \rk M$  whenever $G[X,Y]$ is a complete bipartite graph,}
\end{equation} where $\S(U)$ is the span generated by the elements $x_u$ with $u\in U$. 
The existence of the complete bipartite subgraph $K_{\lfloor t/2\rfloor,\lceil t/2\rceil}$ ensured by Theorem~\ref{thm:matr} implies then the stronger inequality $\rk M\geq\Xind(\Hom(K_2,G))+2$. However, this condition does not sound natural. Furthermore, while whether an assignment of elements from $M$ is an independent representation can be checked in polynomial time in most complexity models (e.g., the independence oracle), checking that such an assignment satisfies the extra condition~\eqref{cond} becomes $\coNP$-complete even for simple cases, e.g., when $M$ is the uniform matroid $U_r^{r-1}$ of rank $r-1$ with $r$ elements (like in Section~\ref{subsec:local}).

\begin{proposition}
Consider an independent representation of a graph over $U_r^{r-1}$. Deciding whether \eqref{cond} is satisfied is $\coNP$-complete.
\end{proposition}

\begin{proof}
We prove that the problem of Proposition~\ref{prop:complex} is reducible to the problem of deciding whether \eqref{cond} is not satisfied.
To this end, we interpret the representation over the matroid  $U_r^{r-1}$ as a proper coloring with $r$ colors such that every vertex has at most $r-2$ colors in its neighborhood. Note that restricting the problem of Proposition~\ref{prop:complex} to this kind of proper colorings is still $\NP$-complete, 
since in any proper $r$-coloring of a graph, it is polynomially checkable whether there is a vertex whose neighbors received  $r-1$ colors.

Let $G$ be a graph with an independent representation of a graph over $U_r^{r-1}$.  
If there is a complete bipartite subgraph $G[X,Y]$ with $\rk\S(X)+\rk\S(Y)\geq r$, then this subgraph has at least $r$ colors, and it is a complete bipartite subgraph with all the colors. If there is a complete bipartite subgraph $G[X,Y]$ with all the colors, then $\rk\S(X)$ and $\rk\S(Y)$ are respectively the number of colors in $X$ and in $Y$ (there are at most $r-2$ colors in $X$ and at most $r-2$ colors in $Y$), and thus $\rk\S(X)+\rk \S(Y)\geq r$.
\end{proof}

For special independent representations, we get this extra condition~\eqref{cond} for free. Orthogonal representations 
are like this. The usual proper coloring is also such an example: as noted in Section~\ref{sec:complexity}, a proper coloring is an independent representation over a uniform matroid of the form $U_r^r$, and 
clearly if $G[X,Y]$ is a complete bipartite graph, then $\rk \S(X\cup Y)\leq \rk M$ is the total number of distinct elements assigned to $X\cup Y$, which is equal to the total number 
of distinct elements assigned to $X$ plus the total number 
of distinct elements assigned to $Y$; condition~\eqref{cond} is satisfied in such a representation. (Note that we recover thus the inequality $\chi(G)\geq\Xind(\Hom(K_2,G))+2$.) Just changing by one unit the rank of the uniform matroid modifies the complexity status of checking~\eqref{cond}.

Here is yet another example where we get \eqref{cond} for free. Let $H$ be a bipartite graph and denote by $U$ one of its sides. Take $M$ as the transversal matroid built from $H$ with elements in $U$: its independent sets are the subsets of $U$ covered by matchings of $H$. Consider an assignment of elements of $M$ to the vertices of $G$ such that adjacent vertices of $G$ get vertices in $U$ with disjoint neighborhoods. This forms an independent representation over $M$ satisfying the condition~\eqref{cond}.
Therefore, for such a matroid, the inequality $\rk M\geq\Xind(\Hom(K_2,G))+2$ holds too.

It is worth noting that if we require in the condition~\eqref{cond} that the cardinalities of $X$ and $Y$ differ by at most $1$, then we still have 
the lower bound of Corollary~\ref{cor:ortho} for independent representations of a graph $G$ over a matroid $M$. 
The problem of deciding whether this new version of condition~\eqref{cond} is satisfied is also $\coNP$-complete; this time because of Corollary~\ref{CoNP:balanced}. 

\subsection{All vectors being pairwise orthogonal}
The intersection of two orthogonal subspaces in a Hilbert space is $\{0\}$. Therefore, if the $s$-dimensional orthogonal representation of $G$ comes from a finite dimensional Hilbert space (like $\R^s$ or $\C^s$ with $\langle\yy,\zz\rangle=\sum_iy_i\overline{z_i}$), then all the vectors assigned to the vertices of the complete bipartite subgraph $K_{\lfloor t/2\rfloor,\lceil t/2\rceil}$ in the statement of Theorem~\ref{thm:matr} are linearly independent. 

A way to improve Theorem~\ref{thm:matr} for the special case of orthogonal representations over $\R$ would be to show that we can find a large bipartite complete subgraph whose assigned vectors are pairwise orthogonal. It is unlikely to get a result in that direction, as shown by the following example.

The $5$-cycle $C_5$ is such that $\od_{\R}(C_5)=\Xind(\Hom(K_2,C_5))+2=3$. Assigning the following $5$ vectors consecutively to the $5$ vertices of $C_5$ leads to a $3$-dimensional orthogonal representation over $\R$ with no $K_{1,2}$ having its $3$ vectors pairwise orthogonal:
$$\left(\begin{array}{c}1 \\ 1 \\ 1\end{array}\right),\quad \left(\begin{array}{c}-1 \\ -1 \\ 2\end{array}\right),\quad \left(\begin{array}{c}3 \\ 1 \\ 2\end{array}\right),\quad \left(\begin{array}{c}-1 \\ 5 \\ -1\end{array}\right),\quad \left(\begin{array}{c}1 \\ 0 \\-1\end{array}\right).$$

\subsection{Gap between the chromatic number and the orthogonal dimension over $\Z_2$}
In the special case of the binary field $\Z_2$, we always have $\od_{\Z_2}(G)=\chi(G)$ when $\chi(G)\leq 6$. Indeed, define $H_t$ to be the graph with vertices the vectors in $\Z_2^t$ with an odd number of nonzero entries, and for which two vertices are adjacent if the corresponding vectors are orthogonal; a $t$-dimensional orthogonal representation of $G$ over $\Z_2$ is a homomorphism from $G$ to $H_t$ and it is easy to check that $\chi(H_t)=t$ when $t\leq 5$. Though, $\chi(H_6)=8$ -- computed with the help of the computer algebra system SageMath. 



Actually, as we explain now, the gap between $\od_{\Z_2}(H_t)$ and  $\chi(H_t)$ can be exponentially large. This a fact has been communicated to us by Ishay Haviv. 

Note first that $\od_{\Z_2}(H_t)=t$ for all $t$: we have clearly  $\od_{\Z_2}(H_t)\leq t$ by definition of $H_t$, and $H_t$ has a clique of size $t$ which implies $\od_{\Z_2}(H_t)\geq t$.
On the other hand, $H_t$ is a graph with $2^{t-1}$ vertices and is an induced subgraph of a $(2^{t-1}-1)$-regular graph $G_t$ with $2^t-1$ vertices and a smallest eigenvalue equal to $-\sqrt{2^{t-2}}$; see~\cite{Alon1997} where this latter graph is denoted by $G(t-1,2)$. Using the Delsarte-Hoffman-type bound for graphs with loops~\cite{LI2014146}, 
we get $\alpha(G_t)\leq O(2^{t/2})$, which implies $\chi(H_t)\geq\frac {|V(H_t)|} {\alpha(H_t)}\geq\frac {|V(H_t)|}{\alpha(G_t)} =\Omega(2^{t/2})$.

The gap between the orthogonal dimension over $\R$ and the chromatic number can be exponentially large as well~\cite[Proposition 2.2]{haviv2019approximating}, and this holds also for any field of characteristic different from $2$ since the proof in the latter paper uses only that property of $\R$. We end by noting that for orthogonal representations over finite fields, a construction similar to that of $H_t$ and similar arguments lead to an exponential gap too.

\section*{Acknowledgments}
The authors thank G\'abor Simonyi for pointing out the paper by Ishay Haviv~\cite{haviv2019topological} that has been at the origin of the present work. They are also grateful to Ishay Haviv for many helpful comments and especially for sharing with them the characterization of the minrank given by Lemma~\ref{lem:ind-minrk}: it contributed to improve the paper.

This work was realized when the first author was visiting \'Ecole des Ponts, from which he acknowledges support. 

\bibliographystyle{plain}
\bibliography{OrthogonalRep}

\end{document}